\documentclass[11pt]{amsart}
  \usepackage{amsmath}
\usepackage{amssymb}
\usepackage{amsfonts}
 \usepackage{eucal}
    \usepackage{xcolor}       \makeatletter
     \def\section{\@startsection{section}{1}%
     \z@{.7\linespacing\@plus\linespacing}{.5\linespacing}%
     {\bfseries
     \centering
     }}
     \def\@secnumfont{\bfseries}
     \makeatother
  \usepackage{a4wide}

   \newtheorem{theorem}{Theorem}[section]
\newtheorem{lemma}[theorem]{Lemma}
\newtheorem{corollary}[theorem]{Corollary}
\newtheorem{proposition}[theorem]{Proposition}

\theoremstyle{definition}

\newtheorem{remark}[theorem]{Remark}

\numberwithin{equation}{section}

\def \a{{\alpha}}
\def \b{{\beta}}

\def \d{{\delta}}
\def \e{{\varepsilon}}

\def \k{{\kappa}}
\def \l{{\lambda}}

\def \p{{\varphi}}
\def \t{{\vartheta}}
\def \m{{\mu}}
\def \s{{\sigma}}

\def \B{{\mathcal B}}
\def \qq{{\qquad}}

\def\beq{\begin{equation}}
\def\eeq{\end{equation}}
  \def\ben{\begin{eqnarray}}
\def\een{\end{eqnarray}}
  
\def\E{{\mathbb E \,}}

\def\P{{\mathbb P}}

\def\R{{\mathbb R}}

\def\Z{{\mathbb Z}}

\def\N{{\mathbb N}}

 \scrollmode
  \font\sevenrm= cmr10 at 7 pt

at  8 pt  
 
\def\ddate {\sevenrm \ifcase\month\or January\or
February\or March\or April\or May\or June\or July\or
August\or September\or October\or November\or December\fi\! {\the\day}, \!{\sevenrm\the\year}}

\date{\today}

\title[$\textit{A strenghtened asymptotic  uniform distribution property}$]
{Strenghtened asymptotic  uniform distribution property and  Rozanov's Theorem} 
\title[$\textit{Rozanov's Theorem and strenghtened asymptotic  uniform distribution}$]
{On  Rozanov's Theorem and   strenghtened asymptotic  uniform distribution} 
  
\begin{document}
  \author{Michel  J.\,G. WEBER}
\address{IRMA, UMR 7501, Universit\'e
Louis-Pasteur et C.N.R.S.,   7  rue Ren\'e Descartes, 67084
Strasbourg Cedex, France.
   E-mail:    {\tt  michel.weber@math.unistra.fr}}
    

\begin{abstract}    For sums $S_n=\sum_{k=1}^n X_k$, $n\ge 1$ of independent random variables    $  X_k  $   taking values in
$\Z$
we prove, as a consequence of a more general result,  that 
if (i) For some function $1\le \phi(t)\uparrow \infty $ as $t\to \infty$, and some constant $C$, we have for all $n$ and $\nu\in \Z$,
 \begin{equation*}\label{abstract1}
\big|B_n\P\big\{ S_n=\nu\big\}- {1\over   \sqrt{ 2\pi } }\ e^{-
{(\nu-M_n)^2\over  2 B_n^2} }\big|\,\le \,  {C\over \,\phi(B_n)},
 \end{equation*}
then (ii) There exists a numerical constant $C_1$, such that for   all    $n $ such that $B_n\ge 6$,   all $h\ge 2$, 
  and $\m=0,1,\ldots, h-1$,
\begin{align*}\label{abstract1}
   \Big|{\mathbb P}\big\{ S_n\equiv\,  \m\ \hbox{\rm{ (mod $h$)}}\big\}- \frac{1}{h}\Big|
  \le   {1\over \sqrt{2\pi}\,  B_n   }+\frac{1+ 2 {C}/{h}
}{ \phi(B_n)^{2/3} } 
+  C_1 \,e^{-(1/ 16 )\phi(B_n)^{2/3}}.
\end{align*}  
Assumption (i) holds if a local limit theorem in the usual form is applicable, and  (ii)  yields a  strenghtening of  Rozanov's necessary condition. 

      Assume in place of (i) that $\t_j =\sum_{k\in \Z}{\mathbb P}\{X_j= k\}\wedge{\mathbb P}\{X_j= k+1 \} >0$,  for each $j$ and that 
 $\nu_n =\sum_{j=1}^n \t_j\uparrow \infty$.
We prove   strenghtened forms of the asymptotic  uniform distribution property.  (iii) Let $\a\!>\!\a'\!>\!0$, $0\!<\!\e\!<\!1$. Then for   each $n$ such that 
$$|x|\le\frac12 \big(  \frac{ 2\a\log (1-\epsilon)\nu_n}{ (1-\epsilon)\nu_n }\big)^{1/2}\qq \Rightarrow \qq{\sin x\over
x}\ge (\a^\prime/\a)^{1/2},$$
 we have
 \begin{eqnarray*} \sup_{u\ge 0}\,\sup_{d<  \pi  
 (   {(1-\epsilon)\nu_n \over 2\a\log (1-\epsilon)\nu_n})^{1/2}
 } \ \big| \P \{d|S_n+u  \}   -  {1\over d} \big|
    \,\le \,2 \,e^{- \frac{\epsilon^2 }{2}\nu_n}+
 \,\big( (1-\epsilon)\nu_n\big)^{-\a'}  .
\end{eqnarray*}
(iv) Let $0<\rho<1 $ and   $0<\e<1$. The sharper uniform bound $2 e^{- \frac{\epsilon^2 }{2}\nu_n}+e^{- ( (1-\epsilon)\nu_n)^\rho}$ is also proved (for a corresponding $d$-region of divisors),   for each   $n$ such that $$|x|\le\frac12 \,\big(  \frac{ 2 }{ ((1-\epsilon)\nu_n)^{1-\rho} }\big)^{1/2}\qq \Rightarrow \qq{\sin x\over
x}\ge \sqrt{1-\e}.$$     \end{abstract}

 \maketitle 
 

\section{\bf  Local limit theorem and asymptotic  uniform distribution.}\label{s1}

  Let $X=\{X_i , i\ge 1\}$ be  a sequence of  independent  variables taking values in $\Z$, 
and let $S_n=\sum_{k=1}^n X_k$, for each $n$. 
\vskip 3 pt  The sequence  $X$
 is said to be {asymptotically uniformly distributed with respect
to lattices of span $d$}, in short a.u.d.($d$), if for $m = 0,1,\ldots,d-1$, we
have
 \beq \label{aud1}\lim_{n\to \infty}\ \P\{S_n \equiv m \,{\rm (mod)}\,d\}=\frac1d.
 \eeq
Equivalenty  for $m = 0,1,\ldots,d-1$, we
have 
 \begin{equation}\label{uad.lim1}\lim_{n\to \infty}\ \P\{d|S_n-m\}=\frac1d.
 \end{equation}  The sequence $X$ is  {asymptotically uniformly distributed}, in short a.u.d., if \eqref{aud1}  holds true for any $d\ge 2$ and $m = 0,1,\ldots,d-1$.

\vskip 5 pt 
\vskip 3 pt 
   
Dvoretsky and Wolfowitz \cite{DW} proved the following characterization. Assume that $X$ is composed with independent random variables taking   only the values 
$$ 0, 1,\ldots, h-1.$$
  In order that the partial sums $\{S_n, n\ge 1\}$  be a.u.d.($h$), it is necessary and sufficient that
\beq \label{aud.dw.ns} \prod_{k=1}^\infty\bigg( \sum_{m=0}^{h-1}\P\{X_k=m\}\,e^{\frac{2i\pi }{h}rm}\bigg) \,=\, 0, \qq \quad (r=1,\ldots, h-1).
 \eeq
Equivalently,
\beq \label{aud.dw.ns.} \prod_{k=1}^\infty\big(\E  e^{\frac{2i\pi }{h}rX_k}\big) \,=\, \lim_{N\to \infty}  \big(\E  e^{\frac{2i\pi }{h}rS_N}\big) \,=\, 0, \qq \quad (r=1,\ldots, h-1).
 \eeq

 This notion plays an important role in the study of the local limit theorem. Let us assume   that the random variables $X_k$ take values in a common lattice $\mathcal L(v_{0},D )$, namely defined by the
 sequence $v_{ k}=v_{ 0}+D k$, $k\in \Z$,   $v_{0} $ and $D >0$ being  reals, and are  square integrable,  
 and let     
 \beq \label{MnBn}M_n= {\mathbb E\,} S_n , \qq B_n^2={\rm Var}(S_n)\to \infty.
 \eeq

\vskip 20pt
We say  that the local limit theorem (in the usual form) is applicable to   $X$    if
 \begin{equation}\label{def.llt.indep}    \sup_{N=v_0n+Dk }\Big|B_n\, {\mathbb P}\{S_n=N\}-{D\over  \sqrt{ 2\pi } }e^{-
{(N-M_n)^2\over  2 B_n^2} }\Big| = o(1), \qq \quad n\to\infty.
\end{equation}
When the random variables $X_i$ are identically distributed,    \eqref{def.llt.indep} reduces to   \begin{equation}\label{llt.iid}    \sup_{N=v_0n+Dk }\Big|  \s \sqrt{n}\, {\mathbb P}\{S_n=N\}-{D\over  \sqrt{ 2\pi } }e^{-
{(N-n\m)^2\over  2 n\s^2} }\Big| = o(1),
\end{equation}
where  $\m={\mathbb E\,} X_1$, $\s^2={\rm Var}(
X_1)$. By Gnedenko's Theorem \cite{G}, see also \cite{P}, p.\,187,  \cite{SW}, Th.\,1.4, \eqref{llt.iid}   holds
if and only if the span $D$ is maximal (there are no  other real numbers
$v'_{0}
$ and
$D' >D$ for which
${\mathbb P}\{X
\in\mathcal L(v'_0,D')\}=1$).

Note that  the transformation
\begin{equation}\label{llt.transf.}
 X'_j= \frac{X_j-v_0}{D},
  \end{equation}
allows one to reduce  to the case  $v_0=0$, $D=1$.
 
\begin{remark}Note that   the series (in $k$)
 \begin{equation}\label{def.llt.indep.sum}    \sum_{N=v_0n+Dk }  \Big( {\mathbb P}\{S_n=N\}-{D\over  \sqrt{ 2\pi } B_n}e^{-
{(N-M_n)^2\over  2 B_n^2} } \Big),
\end{equation}
is  obviously convergent, whereas  
nothing can be deduced concerning its   order  from the very definition of the local limit theorem. Further by using Poisson summation formula
the series associated to the second summand  verifies
 \begin{equation}
 \label{def.llt.indep.poisson} 
  \sum_{N=v_0n+Dk } {D\over  \sqrt{ 2\pi } B_n}e^{-
{(N-M_n)^2\over  2 B_n^2} }\,=\,\sum_{\ell \in\Z} e^{2i\pi \ell \{\frac{v_0n-M_n}{D}\}-\frac{2\pi^2\ell^2 B_n^2}{D^2}},
\end{equation}
and so is 
$ 1+\mathcal O(D/B_n)$, whereas the one associated to the first  is   1. Therefore
 \begin{equation}\label{def.llt.indep.sum.}    \sum_{N=v_0n+Dk } 
 \Big(  {\mathbb P}\{S_n=N\}-{D\over  \sqrt{ 2\pi } B_n}e^{-
{(N-M_n)^2\over  2 B_n^2} }\Big)\,=\, \mathcal O(D/B_n).
\end{equation}
 \end{remark}
 When a strong local limit theorem with convergence in variation holds   we  have the more informative result 
 \beq \label{sllt1}
 \lim_{n\to\infty}
  \sum_{N=v_0n+Dk }\Big|  {\mathbb P}\{S_n=N\}-{D\over  \sqrt{ 2\pi }B_n }e^{-
{(N-M_n)^2\over  2 B_n^2} }\Big| =0.
 \eeq

\vskip 20pt The following result is well-known.
\begin{theorem}[Rozanov] \label{l1}    Let $X=\{X_i , i\ge 1\}$ be  a sequence of  independent  variables taking values in $\Z$,  
and let $S_n=\sum_{k=1}^n X_k$, for each $n$.  The local limit theorem is applicable to   $X$  only if $X$ satisfies the a.u.d. property.
 \end{theorem}
 \begin{remark}    In Petrov \cite{P}, Lemma 1,\,p.\,194, also in   Rozanov's \cite{Ro} Lemma 1,\,p.\,261,  Theorem \ref{l1} is   stated  under the  assumption that a  local limit theorem  in the strong form holds, which is not necessary.
 \end{remark}
We will in fact prove the following stronger result providing an explicit link between the local limit theorem and  the a.u.d. property, through a     quantitative   estimate of  the difference ${\mathbb P} \{ S_n\equiv\,  m\! \hbox{\rm{ (mod $h$)}} \}- {1}/{h}$.

\begin{theorem} \label{l1a}    Let $X=\{X_i , i\ge 1\}$ be  a sequence of  independent  variables taking values in $\Z$,  
and let $S_n=\sum_{k=1}^n X_k$, for each $n$. Assume that 
for some function $1\le \phi(t)\uparrow \infty $ as $t\to \infty$, and some constant $C$, we have for all $n$ 
 \begin{equation}\label{phi.cond}
\sup_{m\in \Z}\Big|B_n\P\big\{ S_n=m\big\}- {1\over   \sqrt{ 2\pi } }\ e^{-
{(m-M_n)^2\over  2 B_n^2} }\Big|\,\le \,  {C\over \,\phi(B_n)}.
 \end{equation}
\vskip 3 pt \noindent Then there exists a numerical constant $C_1$, such that for all $0<\e \le 1$, all  $n $ such that $B_n\ge 6$, and all $h\ge 2$, 
\begin{align*}
\sup_{\m=0,1,\ldots, h-1} \,&\Big|{\mathbb P}\big\{ S_n\equiv\, \m\ \hbox{\rm{ (mod $h$)}}\big\}- \frac{1}{h}\Big|
\cr   &\le   {1\over \sqrt{2\pi}\,  B_n   }+\frac{2C}{h\,\sqrt{\e}\,\phi(B_n)}
+  {\mathbb P}\Big\{   \frac{|S_n -M_n |}{B_n}> \frac{1}{\sqrt \e}\Big\}+C_1 \,e^{-1/(16\e)}.
\end{align*}
\end{theorem}
  \vskip 8 pt 
\begin{remark} \label{rem.thl1a}  It follows from the proof that $C_1=2e\sqrt{\pi}$ is suitable.
\end{remark}
   Choosing $\e= \phi(B_n)^{-2/3}$ and using Tchebycheff's inequality, we get the following
 \begin{corollary}\label{cor}For all   $n $ such that $B_n\ge 6$, and all $h\ge 2$, 
we have 
\begin{align}\label{eps.phi}
\sup_{\m=0,1,\ldots, h-1} \,  \Big|{\mathbb P}\big\{ S_n\equiv\,  \m\ \hbox{\rm{ (mod $h$)}}\big\}- \frac{1}{h}\Big|
 \le H_n ,
\end{align}
with
\begin{align}\label{eps.phi.Hn}
   H_n=  {1\over \sqrt{2\pi}\,  B_n   }+\frac{1+ 2 {C}/{h}
}{ \phi(B_n)^{2/3} } 
+  C_1 \,e^{-(1/ 16 )\phi(B_n)^{2/3}}.
\end{align}
\end{corollary}
 Theorem \ref{l1a} contains Theorem \ref{l1},  since by  definition such a function $\phi$ exists   if the local limit theorem is applicable to   $X$. Further condition \eqref{phi.cond} implies that the local limit theorem is applicable to   $X$.

\begin{remark} Examples of LLT's with speed of convergence are given in   Appendix.
\end{remark}
\begin{proof}
By assumption, 
\begin{equation*}
\Big|B_n\P\big\{ S_n=m\big\}- {1\over   \sqrt{ 2\pi } }\ e^{-
{(m-M_n)^2\over  2 B_n^2} }\Big|\,\le \,  {C\over \phi(B_n) },
 \end{equation*}
 for all $m$ and $n$.  Let  $\e>0$.
We have 
\begin{eqnarray*}
\Big|{\mathbb P}\big\{ S_n\equiv\, m\ \hbox{\rm (mod $h$)}\big\}-  \sum_{|k-M_n|\le B_n/\sqrt \e
\atop k\equiv m\, (h)} {\mathbb P}\big\{S_n=k\}\Big|&\le &
 {\mathbb P}\Big\{   \frac{|S_n -M_n |}{B_n}> \frac{1}{\sqrt \e}\Big\}  
,
\end{eqnarray*}
\begin{align*}
  \Big| \sum_{|k-M_n|\le B_n/\sqrt \e
\atop k\equiv m\, (h)} {\mathbb P}\big\{S_n=k\}-  & {1\over  \sqrt{ 2\pi }B_n }      \sum_{|k-M_n|\le B_n/\sqrt \e
\atop k \equiv m\, (h)}   e^{-
{(k-M_n)^2\over  2 B^2_n} }   \Big| \cr &\le  \,  {C\over B_n\phi(B_n) }\,\sum_{|k-M_n|\le B_n/\sqrt \e
\atop k\equiv m\, (h)}1
\ \le   \,  \frac{2C}{h\,\sqrt{\e}\,\phi(B_n)} 
.
\end{align*}


Letting $z_n= \lfloor M_n\rfloor$, we have
\begin{eqnarray*} \sum_{k\in\Z\atop|k-M_n|> B_n/\sqrt \e } e^{-
{(k-M_n)^2\over  2 B^2_n} }
&\le & \sum_{Z\in \Z \atop |Z-z_n |> B_n/ \sqrt \e } e^{-
{(Z-z_n)^2\over  2 B^2_n} }.
\end{eqnarray*}
Now using the elementary inequality $(a+b)^2\le 2(a^2+b^2)$ for  reals $a$, $b$, we have $|Z-z_n |\le\sqrt 2(  |Z  |+|z_n|) $ and $|Z-z_n |^2\ge   |Z |^2/2-z_n^2$. We can thus continue as follows
\begin{eqnarray*}\,\le\, \sum_{Z\in \Z \atop \sqrt 2(  |Z  |+|z_n|) > B_n/ \sqrt \e } e^{-
{(Z-z_n)^2\over  2 B^2_n} }
&\le& e^{ 
{1\over  2 B^2_n} }\,\sum_{ Z\in \Z \atop  |Z  |  > (B_n/  \sqrt{2 \e}) -1} e^{-
{ Z  ^2\over  4 B^2_n} }
.\qq \end{eqnarray*}

 Assume that $B_n\ge \max( 1/\sqrt 2,4\sqrt{2 \e})$, then ${B_n\over  \sqrt{2 \e}}-2\ge {B_n\over  2\sqrt{2 \e}}$.
In particular  $|Z|\ge 1$ in the previous series, and so we have the estimates
 \begin{eqnarray*}\ \le\ 2\,e^{ 
{1\over  2 B^2_n} }\,\sum_{    Z    > (B_n/  2\sqrt{2 \e}) +1} e^{-
{ Z  ^2\over  4 B^2_n} }&\le & 2\,e
\sum_{    Z    > (B_n/  2\sqrt{2 \e}) +1} \int_{Z-1}^Z e^{-{t^2\over 4 B^2_n}} {\rm d} t
\cr &\le & 2\,e
\int_{B_n/ 2\sqrt{2 \e} }^\infty e^{-{t^2\over 4 B^2_n}} {\rm d} t
\cr(
t= \sqrt 2B_n u)\quad&=& 2 \sqrt{2 }e
B_n\int_{1/4\sqrt{  \e}}^\infty e^{-{u^2\over 2 }} {\rm d} u
\cr &\le &  2\sqrt{2 }e
B_n\sqrt{{\pi\over 2}}\, e^{-1/(16\e)}
\cr &= &
2e\sqrt{\pi}
B_n \, e^{-1/(16\e)},\end{eqnarray*}
since  $  e^{x^2/2}\int_x^\infty e^{-t^2/2}{\rm d}t  \le \sqrt{{\pi\over 2}}$,   for any $x\ge 0$.
  
Therefore
\begin{eqnarray}\label{est.1}
& &\Big|{\mathbb P}\big\{ S_n\equiv\, m\ \hbox{\rm (mod $h$)}\big\}
 - {1\over  \sqrt{ 2\pi }B_n }    \sum_{ k \equiv m\, (h)}  e^{-
{(k-M_n)^2\over  2 B^2_n} }\Big|\cr &\le & {\mathbb P}\Big\{   \frac{|S_n -M_n |}{B_n}> \frac{1}{\sqrt \e}\Big\}  + \frac{2C}{h\,\sqrt{\e}\,\phi(B_n)}   +C_1 \, e^{-1/(16\e)} ,
\end{eqnarray}
with $C_1=2e\sqrt{\pi}$.
\vskip 5 pt 
Recall   Poisson summation formula: for   $x\in \R,\  0\le
\d\le 1
$, 
\begin{equation}\label{poisson}\sum_{\ell\in \Z} e^{-(\ell+\d)^2\pi x^{-1}}=x^{1/2} \sum_{\ell\in \Z}  e^{2i\pi \ell\d -\ell^2\pi x}.
 \end{equation}  
  Write $k=m+l h$, $M'_n=M_n-m$, 
\begin{equation}{(k-M_n)^2\over  2 B^2_n}={( l h-M'_n)^2\over  2 B^2_n}={( l  -\lceil M'_n/h\rceil+\{ M'_n/h\})^2\over  2 B^2_n/h^2}={( \ell  +\{ M'_n/h\})^2\over  2 B^2_n/h^2},
\end{equation}
letting $ \ell=l  -\lceil M'_n/h\rceil$. 
\vskip 3 pt
By applying it with $x=2 B^2_n\pi /h^2$, $\d=\{ M'_n/h\}$, we get 
\begin{equation}     \sum_{ k \equiv m\, (h)}   e^{-
{(k-M_n)^2\over  2 B^2_n} }\,=\, \sum_{\ell \in \Z}e^{-{( \ell  -\{ M'_n/h\})^2\over  2 B^2_n/h^2}}\,=\,  {\sqrt{2 \pi}B_n\over h}\,\sum_{\ell \in \Z}e^{ -2i\pi \ell \{ M'_n/h\} -2\pi^2B_n^2\ell^2/h^2}.
\end{equation}
 Whence
\begin{equation}   \Big|{  h\over \sqrt{2 \pi}B_n}  \sum_{ k \equiv m\, (h)}   e^{-
{(k-M_n)^2\over  2 B^2_n} } -1\Big|\le  \sum_{|\ell |\ge 1}e^{ -2\pi^2B_n^2\ell^2/h^2}.
\end{equation}
But for  any positive real  $a$,
 \begin{equation}\label{aux.est1}\sum_{H=1}^\infty e^{-aH^2}\le {\sqrt \pi\over 2 }\min({1\over \sqrt a}, {1\over a}).
 \end{equation}
Therefore with $a= 2\pi^2B_n^2/h^2$,
\begin{equation*}   \Big|{  h\over \sqrt{2 \pi}B_n}  \sum_{ k \equiv m\, (h)} 
  e^{-{(k-M_n)^2\over  2 B^2_n} } -1\Big|\le  {\sqrt \pi }\min({h\over \sqrt{2}\pi B_n   }, {h^2\over 2\pi^2B_n^2 })\le {h\over \sqrt{2\pi}\,  B_n   }.
\end{equation*}
We have thus obtained the explicit bound 
\begin{equation}\label{est2}   \Big|{  1\over \sqrt{2 \pi}B_n}  \sum_{ k \equiv m\, (h)}   e^{-
{(k-M_n)^2\over  2 B^2_n} } -\frac1h\Big| \le {1\over \sqrt{2\pi}\,  B_n   }.
\end{equation}
By carrying it back to \eqref{est.1}, we get for any $\e>0$, all  $n $ such that $B_n\ge \max( 1/\sqrt 2,4\sqrt{2 \e})$, and all $h\ge 2$, 
  
\begin{align}\label{est.3}\sup_{\m=0,1,\ldots, h-1}
\Big|{\mathbb P}\big\{ S_n\equiv\, \m\ \hbox{\rm{ (mod $h$)}}\big\}-\frac{1}{h}\Big|
 &\le   {1\over \sqrt{2\pi}\,  B_n   }+ \frac{2C}{h\,\sqrt{\e}\,\phi(B_n)}
\cr   &\quad+  {\mathbb P}\Big\{   |S_n -M_n |> {B_n\over \sqrt \e}\Big\}+ C_1\, e^{-1/(16\e)}.
\end{align}
This is fulfilled if we choose $0<\e \le 1$, and $n$ such that $B_n\ge 6$, whence the claimed estimate.
\end{proof}



 \section{Local limit theorem in the strong form}\label{s2}
There are easy examples of sequences $X$ for which the fulfilment of the local limit theorem  depends on the behavior of the first members of $X$. 
Hence  it is reasonable to introduce the following definition due to Prohorov \cite{Pr}. A local limit theorem   in  the {\it strong form}
(or {\it in a strengthened form})
 is said to be
applicable to $  X$, if a local limit theorem in the usual form   is applicable to any subsequence extracted from $  X$, which differs from
$  X$ only in a finite number of members. 
\vskip 3 pt This definition can be made a bit more convenient, see Gamkrelidze \cite{Gam3}.
Let
\begin{equation}S_{k,n}=X_{k+1}+ \ldots +X_{k+n},\qq A_{k,n}=\E S_{k,n}, \qq B^2_{k,n}={\rm Var} (S_{k,n}).
\end{equation}
The local limit theorem   in  the  strong form  holds if and only if
\begin{equation}\label{lltsf.ref}
\P\big\{ S_{k,n}=m\big\}= {D\over B_{k,n} \sqrt{ 2\pi } }\ e^{-
{(m-A_{k,n})^2\over  2 B_{k,n}^2} }+o\Big({1\over B_{k,n}  }\Big),
\end{equation}
uniformly in $m$ and every finite $k$, $k=0,1,2, \ldots$, as $n\to \infty$ and $B_{k,n}\to \infty$.
\vskip 8 pt
Rozanov's   necessary condition     states as follows.

\begin{theorem}[\cite{Ro},\,Th.\,I]\label{rozanov.I} Let $  X= \{ X_j , j\ge 1\}$ be a sequence of independent, square integrable random variables taking  values in $\Z$. Let $b_k^2= {\rm Var}(X_k)$, $B_n^2 =b_1^2+\ldots+ b_n^2$. Assume that
\begin{equation}\label{Ro.A}B_n \to \infty\qq\qq {\rm as} \ n\to \infty.
\end{equation}
The following condition is necessary for the applicability of a local limit theorem  in the strong form to the sequence $X$,
\begin{equation}\label{rozn}\prod_{k=1}^\infty
\big[ \max_{0\le m< h} {\mathbb P}\big\{X_k\equiv m\, {\rm (mod {\it \, h })} \big\}\big]=0\qq {\it for\ any\ }  h\ge 2 .
\end{equation}
\end{theorem}
\vskip 2 pt 
Condition \eqref{rozn} is also sufficient in some important examples, in particular if    $X_j$ have stable limit distribution, see Mitalauskas  \cite{Mit}. 
We briefly indicate how  Theorem \ref{rozanov.I} is proved. 
If  the   local limit theorem    in the strong form  is applicable to the sequence $X$, then
\begin{equation} \label{roz.cond}
\sum_{k=1}^\infty \  {\mathbb P}\big\{ X_k\not\equiv 0\, ({\rm mod} \ h) \big\}= \infty, \qq {\it for\ any} \ h\ge 2.
\end{equation}
Indeed, otherwise given $h\ge 2$,  by the  Borel--Cantelli lemma,   on a set of measure greater than $3/4$, $X_k\equiv 0\, ({\rm mod} \ h)$ for all $k\ge k_0$, say. The new sequence $X'$ defined by $X'_k=0$ if $k< k_0$, $X'_k=X_k$ unless, with partial sums $S'_n$, verifies ${\mathbb P}\{ S'_n\equiv 0\, ({\rm mod} \ h) \}>3/4$ for all $n$ large enough, and this can be used to bring a contradiction with the fact
that ${\mathbb P}\{ S'_n\equiv 0\, ({\rm mod} \ h) \}$ should converge to $1/h$.  
\vskip 2 pt  

\vskip 3 pt 
 
The arithmetical quantity 
$$\max_{0\le m< h} {\mathbb P}\big\{X_k\equiv m\, {\rm (mod({\it h})} \big\}$$
 also appears in the study of 
  local limit theorems with arithmetical sufficient  conditions. The   approaches used   (Freiman,  Moskvin  and Yudin \cite{FMY}, Mitalauskas \cite{Mit1},  Raudelyunas \cite{Rau} and later   Fomin \cite{Fo},    for instance) require  the random variables to do not overly much concentrate in a particular residue class $m$ (mod $h$)  of $\Z$, and impose 
arithmetical conditions of type: For all $h\ge 2$
\begin{equation}\label{llt.arithm.cond.}
  \max_{0\le m<h}{\mathbb P}\{X_k\equiv m \ {\rm (mod\, {\it h})}\}\le 1-\a_k,
  \end{equation}
for all $k$, where $\a_k$ is some specific sequence of reals decreasing to $0$. In addition, one generally have that $\sum _k \a_k = \infty$. Although   the simple form of local limit theorem is here considered, for obvious reasons,     condition \eqref{rozn} brings nothing more in this context.   

\vskip 3 pt

\vskip 3 pt As a consequence of   the quantitative formulation of the a.u.d. property  obtained in Theorem \ref{l1a}, we have  the following result.
\begin{theorem} Under the assumptions of Theorem \ref{rozanov.I}, assume further that the local limit theorem   is applicable to a sequence $X$.  Then
\vskip 3 pt
  {\rm(i)}  
 \begin{eqnarray*} \limsup_{h\to \infty}\ \prod_{k=1}^\infty \max_{0\le m< h}\P\{X_k\equiv\, m\ \hbox{\rm (mod $h$)}\}\,=\ 0.
\end{eqnarray*} 
 
 {\rm(ii)} There exists a  function $1\le \phi(t)\uparrow \infty $ as $t\to \infty$,  such that 
\begin{eqnarray*}     \sum_{k=1}^n\frac{  \max_{0\le m< h}\P\{X_k\equiv\, m\ \hbox{\rm (mod $h$)}\}}{1- \max_{0\le m< h}\P\{X_k\equiv\, m\ \hbox{\rm (mod $h$)}\}}
 &\ge &  -\log\big(\frac{1}{h}+H_n\big),
\end{eqnarray*}
where 
$H_n={  1\over \sqrt{2\pi}\,  B_n   }+\frac{1+ 2 {C}/{h}}{ \phi(B_n)^{2/3} } 
+  C_1\,e^{-(1/ 16 )\phi(B_n)^{2/3}} $, and   $C,C_1$ are absolute constants.\end{theorem}
\begin{proof} We purpose a direct argument. 
Consider a  
 sequence $Y$ where $Y_k=X_k-m_k$, $m_k$ are   integers, for all $k\ge 1$. Let $h\ge 2$ be fixed. Choose $m_k$ so that
 $$\max_{0\le m< h} {\mathbb P}\big\{X_k\equiv m\, {\rm mod({\it h})} \big\}= {\mathbb P}\big\{X_k\equiv m_k\, {\rm mod({\it h})} \big\}
 =  {\mathbb P}\big\{Y_k\equiv 0\, {\rm mod({\it h})} \big\}
 , $$
and let $\m_n=\sum_{k=1}^nm_k$.  Note that $\sum_{k=1}^n  Y_k=S_n -\m_n$, ${\rm Var}(\sum_{k=1}^nY_k)={\rm Var}(S_n)=B_n^2$.
 \vskip 2 pt \vskip 2 pt As the local limit theorem is   applicable to the sequence $X$,
  condition \eqref{phi.cond} is  satisfied for some function $1\le \phi(t)\uparrow \infty $ as $t\to \infty$, namely we have for all   $n$,
  \begin{equation*} 
\sup_{\nu\in \Z}\Big|B_n\P\big\{ S_n=\nu\big\}- {1\over   \sqrt{ 2\pi } }\ e^{-
{(\nu-M_n)^2\over  2 B_n^2} }\Big|\,\le \,  {C\over \,\phi(B_n)}.
 \end{equation*}
 Given $n$,   letting $\nu=m+\m_n$ and observing that  $\P \{ \sum_{k=1}^n  Y_k =m\}=\P \{ S_n -\m_n =m\}$, we get for $m\in \Z$, $n\ge 1$,
  \begin{equation*}
\Big|B_n\P\Big\{ \sum_{k=1}^n  Y_k =m \Big\}- {1\over   \sqrt{ 2\pi } }\ e^{-
{(m+\m_n-M_n)^2\over  2 B_n^2} }\Big|\,\le \,  {C\over \,\phi(B_n)}.
 \end{equation*} 
Thus $Y$ satisfies condition \eqref{phi.cond} with the same function $\phi(n) $.
  \vskip 2 pt \vskip 2 pt Applying  Remark \ref{rem.thl1a} to the sequence $Y$, it follows that,
\begin{eqnarray}\label{kb} \prod_{k=1}^n \max_{0\le m< h}\P\{X_k\equiv\, m\ \hbox{\rm (mod $h$)}\}&=&\prod_{k=1}^n  \P\{Y_k\equiv\, 0\ \hbox{\rm (mod $h$)}\}
\cr &\le &
{\mathbb P}\big\{ \sum_{k=1}^n  Y_k\equiv\, 0\ \hbox{\rm (mod $h$)}\big\}\le \frac{1}{h}+H_n,
\end{eqnarray} 
where $H_n$ has the form given in the statement, and $H_n\to 0$ as $n\to \infty$.
 \vskip 2 pt Letting $n$ tend to infinity in \eqref{kb} implies,
 \begin{eqnarray}\label{kbh}  \prod_{k=1}^\infty \max_{0\le m< h}\P\{X_k\equiv\, m\ \hbox{\rm (mod $h$)}\}&\le& \frac{1}{h}.
\end{eqnarray}
This being true for each $h$, $h\ge 2$, letting now $h$ tend to infinity in \eqref{kbh} yields,
 \begin{eqnarray} \limsup_{h\to \infty}\ \prod_{k=1}^\infty \max_{0\le m< h}\P\{X_k\equiv\, m\ \hbox{\rm (mod $h$)}\}&=& 0.
\end{eqnarray}

 \vskip 2 pt \vskip 2 pt
We also have by using  the elementary inequality  $\log( 1-x)\ge -x/(1-x)$, $0\le x<1$,
\begin{eqnarray*} \prod_{k=1}^n \P\{Y_k\equiv\, m\ \hbox{\rm (mod $h$)}\}
&=&\prod_{k=1}^n\big(1- \P\{Y_k\not\equiv\, m\ \hbox{\rm (mod $h$)}\}\big) \cr&=&e^{\sum_{k=1}^n \log(1-\P\{Y_k\not\equiv\, m\ \hbox{\rm (mod $h$)}\})}\cr&\ge &e^{-\sum_{k=1}^n \P\{Y_k\not\equiv\, m\ \hbox{\rm (mod $h$)}\}/(1-\P\{Y_k\not\equiv\, m\ \hbox{\rm (mod $h$)}\})}
.
\end{eqnarray*}
  Thus by   Remark \ref{rem.thl1a},
\begin{eqnarray*}     \sum_{k=1}^n\frac{  \max_{0\le m< h}\P\{X_k\equiv\, m\ \hbox{\rm (mod $h$)}\}}{1- \max_{0\le m< h}\P\{X_k\equiv\, m\ \hbox{\rm (mod $h$)}\}}
&\,=\,&\sum_{k=1}^n\frac{ \P\{Y_k\not\equiv\, m\ \hbox{\rm (mod $h$)}\}}{1-\P\{Y_k\not\equiv\, m\ \hbox{\rm (mod $h$)}\}}
\cr&\ge &  -\log\big(\frac{1}{h}+H_n\big).
\end{eqnarray*}
\end{proof}
\begin{remark}\label{kb.rem} (i) Note   that the   bound used in \eqref{kb} is very weak since  
\begin{eqnarray*}   \prod_{k=1}^n  \P\{Y_k\equiv\, m\ \hbox{\rm (mod $h$)}\}
  \ =\
{\mathbb P}\big\{ \forall J\subset [1,n],\ \sum_{k\in J}   Y_k\equiv\, m\ \hbox{\rm (mod $h$)}\big\} .
\end{eqnarray*} 
One can replace individuals $Y_k$ by sums over blocks according to any partition of $\{1,\ldots,n\}$.
\vskip 2 pt \noindent  (ii) Sets of multiples serve as good test sets  for the applicability of the local limit theorem because addition is a closed operation. What can be derived when testing the applicability of the local limit theorem with other remarkable sets of integers (squarefree numbers, primes numbers, power numbers, geometric growing sequences, \ldots) is unknown. 
 Concerning the squarefree integers, namely
having no squared prime factors,  we note the bound \begin{equation}\label{squarefree} \Big|2^{-n}
\sum_{j\,  { \rm squarefree}} C_n^j  - \frac{6}{\pi^2}\Big|\le C_1e^{-C_2  {(\log 
n^{{3/5}}}/{(\log\log  n) ^{1/5}}}.
\end{equation}
   We refer to \cite{DS}.
\end{remark} 
\section{Random sequences satisfying  the a.u.d. property}\label{s3}
It has some interest to relate the a.u.d. property for Bernoulli sums   to the one   of sets having Euler density, in this particular case here, arithmetic progressions. A subset $A$ of $  \N$ is said to have Euler density $\l$ 
with parameter
$\varrho $ (in short ${ E}_\varrho$ density $\l$) if 
\begin{equation*} 
\lim_{n\to \infty}\sum_{j\in A} C_n^j \varrho^j (1-\varrho)^{n-j}= \l.
\end{equation*}
 By a result due to  Diaconis and Stein, we have the following characterization.
   \begin{theorem}[\cite{DS},\,Th.\,1]\label{ds}For any $A\subset \N$, and $\varrho\in ]0,1[$ the following assertions are equivalent:
 \begin{eqnarray*} \label{dschar} ({\rm i})  & &\qq \hbox{\it $A$ has  ${E}_\varrho$ density $\l$},
  \cr   ({\rm ii})  & &\qq\lim_{t\to \infty} e^{-t}\sum_{j\in A} \frac{t^j}{j!}=\l,
 \cr ({\rm iii})  & &\qq \hbox{\it  for all $\e>0$}, \quad    \lim_{n\to \infty} \frac{\#\{j\in A : n\le j< n+\e\sqrt n\} }{\e\sqrt n} 
 =\l . 
 \end{eqnarray*}
\end{theorem} 
Applying (iii) with $\rho=\frac12$, to
\begin{equation}\label{setA}A= \{  u+kd,\ k\ge 1\},
\end{equation} 
straightforwardly implies
\begin{lemma}\label{ber.uad}
Let  $\B_n=\b_1+\ldots+\b_n$,  where   $ \b_i  $ are i.i.d.  Bernoulli random variables. Then $\{\B_n, n\ge 1\}$ is     a.u.d.($d$) for any $d \ge 2$.
\end{lemma}
\vskip 10 pt



Now consider the independent case and introduce the following characteristic. 
 Let $Y$ be
  a random variable with values in $\Z$. Put
  \begin{eqnarray}\label{vartheta}  \t_Y =\sum_{k\in \Z}{\mathbb P}\{Y= k\}\wedge{\mathbb P}\{Y= k+1 \} ,
\end{eqnarray}
where $a\wedge b=\min(a,b)$. Note    that $0\le \t_Y<1$.

 \begin{theorem} \label{t3}  Let $  X= \{ X_j , j\ge 1\}$ be a sequence of independent  random variables taking  values in $\Z$. Assume that   $\t_{X_j}>0$ for each $j$. Further   assume that the series 
 $\sum_{j=1}^\infty \t_{X_j}$ diverges. Then $X$ is     a.u.d., 
  the conclusion holds in particular if the $X_j$ are i.i.d. and $\t_{X_1}>0$.
 \end{theorem} 
   Note that  no integrability condition is required, whereas square integrability is required in order that the local limit theorem be applicable.   
 We   prove in the next section  that  if  the series 
$\sum_{j=1}^\infty \t_{X_j}$ diverges, much more is in fact true.   
Under the assumption made, each $X_j$ admits a Bernoulli component. This is the principle of a   coupling method (the Bernoulli part extraction) introduced by McDonald \cite{M}, Davis and McDonald  \cite{MD}  in the study of the local limit theorem. See    Weber \cite{W} for an application of this method  to almost sure local limit theorem, and   Giuliano and Weber  \cite{GW3} where this method is used to obtain   approximate local limit theorems with effective rate.

 \vskip 3 pt   Before passing to the proof, we briefly recall some facts and state an auxiliary  Lemma. Let    $\mathcal L(v_0,D)$ be a lattice  defined by the
 sequence $v_{ k}=v_{ 0}+D k$, $k\in \Z$,  
 $v_{0} $ and $D >0$ being  real numbers. Let $X$ be
  a random variable   such that  ${\mathbb P}\{X
\in\mathcal L(v_0,D)\}=1$, and assume  that 
$\t_X>0$.  
Let $ f(k)= {\mathbb P}\{X= v_k\}$, $k\in \Z$. Let also $0<\t\le\t_X$. Associate to $\t$ and $X$  a
sequence $  \{ \tau_k, k\in \Z\}$     of   non-negative reals such that
\begin{equation}\label{basber0}  \tau_{k-1}+\tau_k\le 2f(k), \qq  \qq\sum_{k\in \Z}  \tau_k =\t.
\end{equation}
For instance $\tau_k=  \frac{\t}{\nu_X} \, (f(k)\wedge f(k+1))  $ is suitable.
Next   define   a pair of random variables $(V,\e)$   as follows:
  \begin{eqnarray}\label{ve} \qq\qq\begin{cases} {\mathbb P}\{ (V,\e)=( v_k,1)\}=\tau_k,      \cr
 {\mathbb P}\{ (V,\e)=( v_k,0)\}=f(k) -{\tau_{k-1}+\tau_k\over
2}    .  \end{cases}\qq (\forall k\in \Z)
\end{eqnarray}

 \begin{lemma} \label{lemd} Let $L$
be a Bernoulli random variable    which is independent of  $(V,\e)$, and let  $Z= V+ \e DL$. Then $Z\buildrel{\mathcal D}\over{ =}X$.
\end{lemma}

\begin{proof}[Proof of Theorem \ref{t3}]  
    We apply   Lemma \ref{lemd} with $D=1$ to each $X_j$, and choose $0<\t_j\le\t_{X_j}$ so that the series $\sum_{j=1}^\infty \t_j$ diverges.
One can  associate to them a
sequence of independent vectors $ (V_j,\e_j, L_j) $,   $j=1,\ldots,n$  such that
 \begin{eqnarray}\label{dec0} \big\{V_j+\e_j    L_j,j=1,\ldots,n\big\}&\buildrel{\mathcal D}\over{ =}&\big\{X_j, j=1,\ldots,n\big\}  .
\end{eqnarray}

Further the sequences $\{(V_j,\e_j),j=1,\ldots,n\}
 $   and $\{L_j, j=1,\ldots,n\}$ are independent.
For each $j=1,\ldots,n$, the law of $(V_j,\e_j)$ is defined according to (\ref{ve}) with $\t=\t_j$.  And $\{L_j, j=1,\ldots,n\}$ is  a sequence  of
independent Bernoulli random variables. Set
\begin{equation}\label{dec} 
W_n =\sum_{j=1}^n V_j,\qq M_n=\sum_{j=1}^n  \e_jL_j,  \quad B_n=\sum_{j=1}^n
 \e_j .
\end{equation} 
  Denoting again $X_j= V_j+ \e_jL_j$, 
$j\ge 1$, we have
\begin{eqnarray}\label{dep} \P \{d|S_n  +u\}   &=& \E_{(V,\e)}   \,   \P_{\!L}
\Big\{d|\big(   \sum_{j= 1}^n \e_jL_j+W_n  \big)+u
\Big\}
. \end{eqnarray}
As    $\sum_{j= 1}^n \e_jL_j\buildrel{\mathcal D}\over{ =}\sum_{j=1}^{B_n } L_j$, we have
 \begin{eqnarray*}        \P_{\!L}
\Big\{d|\big(  \sum_{j= 1}^n \e_jL_j+W_n  \big)+u
\Big\} &=&    \P_{\!L}
\Big\{d|  \,  \sum_{j=1}^{B_n } L_j+\big(W_n   +u\big)
\Big\}. 
\end{eqnarray*}
  
In view of the dominated convergence theorem, it suffices to prove that for each $d\ge 2$,
 \begin{eqnarray*}       \P_{\!L}
\Big\{d| \,    \sum_{j=1}^{B_n } L_j+ (W_n   +u) \Big\}\  \to \frac{1}{d},
\end{eqnarray*}
as $n\to \infty$, $\P_{(V,\e)}$ almost surely. 
But the set (compare with \eqref{setA}) 
\begin{equation*} A= \{   (W_n   +u)+kd,\ k\ge 1\},
\end{equation*}
now  depends on $W_n$, thus on $n$, which is complicating things. 
However we can write
 \begin{equation*}
 \chi\Big({d\, \big|\,\sum_{j=1}^{B_n } L_j+ (W_n   +u)}\Big)=\frac1d\,\sum_{j=0}^{d-1} e^{2i\pi  {j \over d}  (W_n   +u)}e^{2i\pi  {j \over d}\sum_{j=1}^{B_n } L_j}.
\end{equation*}
By integrating with respect to $ \P_{\!L}$ we get,
 $$ \P_{\!L}
\Big\{d| \,   \sum_{j=1}^{B_n } L_j+\big(W_n   +u\big)
\Big\}={1\over d}+{1\over d}\sum_{j=1}^{d-1} e^{2i\pi  {j \over d}  (W_n   +u)}\big(\cos { \pi j\over d}\big)^{B_n}
 .$$
By the assumption made, 
$B_n$ tends to infinity $\P_{(V,\e)}$ almost surely, ((8.3.5) in \cite{W2} for instance). Thus the latter sum tends to 0 as $n\to \infty$, $\P_{(V,\e)}$ almost surely.
Therefore by the convergence argument invoked before, $ \P \{d|S_n  +u\}$ tends to ${1\over d}$ as $n$ tends to infinity, for any  $d\ge 2$ and $u\in \N$. Whence it follows that the sequence $\{S_n, n\ge 1\}$ is     {\rm a.u.d.}\,.
\end{proof}

\section{Random sequences satisfying a strenghtened    a.u.d. property.}\label{s4}
  For Bernoulli sums, the  a.u.d.  property is   only a rough  aspect of the value distribution of divisors of $\B_{ n}+u$, $u\ge 0$ integer. Much more is known. 
\begin{theorem}[\cite{W3},\, Th.\,2.1]\label{estPdlBn.u} 
   We have the uniform estimate
\begin{equation*}
 \sup_{u\ge 0}\,\sup_{2\le d\le n}\Big|\P\big\{  d|   \B_{ n}+u \big\}-  {1\over d}\sum_{ 0\le |j|< d }
e^{i\pi (2u+n){j\over d}}\  e^{  -n  
{\pi^2j^2\over 2d^2}}\Big|= {\mathcal O}\big((\log n)^{5/2}n^{-3/2}\big).
\end{equation*}
\end{theorem}
The special case $u=0$ was proved in   \cite[Th.\,II]{W1}.  Introduce the Theta function
 \begin{equation}\label{theta.u.}
\Theta_u(d,n)  =  \sum_{\ell\in \Z}  e^{i\pi (2u+n){\ell\over d}}\  e^{  -n  
{\pi^2\ell^2\over 2d^2}}.
 \end{equation}
 By Poisson summation formula
 \begin{equation}\label{theta.u..}
\Theta_u(d,n)   = \Big(d\sqrt{\frac{2}{\pi n}}\Big)\ \sum_{\ell \in \Z} e^{-(\ell+\{\frac{u+n/2}{d}\})^2\frac{2d^2}{n}}.
 \end{equation}
As a consequence of Theorem \ref{estPdlBn.u}, we get\begin{corollary}\label{cor.estPdlBn.u} 
 We have the uniform estimate 
\begin{equation*} \sup_{u\ge 0}\, \sup_{2\le d\le n}\Big|\P\big\{  d|   \B_{ n}+u \big\}-{\Theta_u(d,n)\over d} \Big| \le C \,(\log
n)^{5/2}n^{-3/2} .
\end{equation*} 
\end{corollary}
Apart from this important but specific case,  it seems that the speed of convergence  in the limit \eqref{aud1} was not investigated, in particular when $d$ and $n$ are varying simultaneously. 

\vskip 3 pt 

Consider the independent case and assume as in Theorem \ref{t3}, that  $\nu_n =\sum_{j=1}^n \t_j\uparrow \infty $. The  speed  of uniform convergence over regions  (in $d$ and $n$)   presents  a singularity when $d$ is getting too  close to $\sqrt {\nu_n}$.  That quantity   already appears in Davis and McDonald \cite{MD}. On the other hand   when $d$ is not   close to $\sqrt {\nu_n}$, in a sense that we shall make precise, we show that an explicit  speed of convergence  can be assigned, this under the {\it sole} divergence assumption of the series $\sum_{j=1}^\infty \t_j$. So, for this important class of independent sequences, the  well-known a.u.d. necessary condition turns up to be a particularly weak requirement. Further one can show by using Poisson summation formula that in the Bernoulli case, the local limit theorem implies a weaker speed of convergence  than the one obtained in Theorem \ref{estPdlBn.u}. 
 
 \vskip 3 pt The speed of uniform convergence problem   for {\it all} $d$ and $n$, $n\ge d\ge 2$, $n\to\infty$, is more complicated and one must restrict to the i.i.d. case. In place of the limiting term ${1}/{d}$ appears a more complicated Theta elliptic function. See  \cite{W3}. For the independent case, the approach used becomes inoperant, due to appearance of  integral products with interlaced integrants.
  In fact, what will make possible  to handle the independent case, is not just that $d$ and $\sqrt {\nu_n}$ are not too close, but also that in  background,      symmetries properties of the Bernoulli model permitted   to effect the necessary calculations
  in the first quadrant and {\it not} in the half-circle. This point is crucial   for getting the uniform speed of convergence in Theorem \ref{estPdlBn.u}.  This is  explained in \cite{W3}, see reduction Lemma 2.3. In short, when the Bernoulli extraction part applies, these symmetry properties allow one to get a speed of convergence. The proof in the Bernoulli case is transposable to 
  other systems of random variables when such symmetries exist. This is not the case for  the   Hwang and  Tsai model of the Dickman function \cite{HT}, \cite{GSW}, neither for  the Cram\'er model of primes \cite{W4}. 
  \vskip 3 pt We prove  the following result.
 \begin{theorem}\label{saud1} Assume that $D=1$,   $\t_{X_j}>0$ for each $j$, and that the series 
 $\sum_{j=1}^\infty \t_{X_j}$ diverges.
 Let $\a\!>\!\a'\!>\!0$, $0\!<\!\e\!<\!1$. Then for   each $n$ such that 
$$|x|\le\frac12 \sqrt{  \frac{ 2\a\log (1-\epsilon)\nu_n}{ (1-\epsilon)\nu_n }}\qq \Rightarrow \qq{\sin x\over
x}\ge (\a^\prime/\a)^{1/2},$$
recalling that $\nu_n =\sum_{j=1}^n \t_j$,  we have
 \begin{eqnarray*} \sup_{u\ge 0}\,\sup_{d<  \pi  \sqrt{   (1-\epsilon)\nu_n \over 2\a\log (1-\epsilon)\nu_n}} \ \Big| \P \{d|S_n+u  \}   -  {1\over d} \Big|
    &\le &2 \,e^{- \frac{\epsilon^2 }{2}\nu_n}+
 \,\big( (1-\epsilon)\nu_n\big)^{-\a'}  .
\end{eqnarray*}
\end{theorem}
\
For the proof we   use the following Lemma.
\begin{lemma}[\cite{di}, Theorem 2.3] \label{di.1}
Let $X_1, \dots, X_n$     be independent random variables, with $0 \le X_k \le 1$ for each $k$.
Let $S_n = \sum_{k=1}^n X_k$ and $\mu = \E S_n$. Then for any $\epsilon >0$,
 \begin{eqnarray*}
{\rm (a)} &&
 \P\big\{S_n \ge  (1+\epsilon)\mu\big\}
    \le  e^{- \frac{\epsilon^2\mu}{2(1+ \epsilon/3) } } .
\cr {\rm (b)} & &\P\big\{S_n \le  (1-\epsilon)\mu\big\}\le    e^{- \frac{\epsilon^2\mu}{2}}.
 \end{eqnarray*}
\end{lemma}
We also need the following result. 
\begin{proposition}[\cite{W3}, Corollary 2.4] 
  \label{special.cases}{\rm (i)} For each $\a\!>\!\a'\!>\!0$ and $n$ such that $  \tau_n\ge (\a^\prime/\a)^{1/2}$, where 
  \begin{equation*} 
  \tau_n= {\sin\p_n/2\over
\p_n /2}, \qquad\qquad \p_n=   \big( {2\a\log n \over n}\big)^{1/2},
\end{equation*}
we have \begin{equation*}
 \sup_{u\ge 0}\,\sup_{d<  \pi   \sqrt{   n \over 2\a\log n}}\Big|\P\big\{  d|   \mathcal B_{ n} +u\big\}-{1\over d} 
\Big|\,\le\, n^{-\a'}.
\end{equation*}
{\rm (ii)}
Let $0<\rho<1 $.  Let also $0<\eta<1$, and suppose $n$ sufficiently large so that $\widetilde\tau_n\ge \sqrt{1-\eta}$, where
$$  \widetilde\tau_n= {\sin\psi_n/2\over
\psi_n /2}\qq \qq \psi_n= \big({2n^\rho \over n}\big)^{1/2}.$$ 
Then,
\begin{equation*}  \sup_{u\ge 0}\,\sup_{d<  (\pi/\sqrt 2) n^{(1-\rho)/2} }\Big|\P\big\{  d|   \mathcal B_{ n} +u\big\}-{1\over d} 
\Big|\,\le\, e^{-(1-\eta)\, n^\rho}.
\end{equation*}
  \end{proposition} 
 \begin{proof}[Proof of Theorem \ref{saud1}]
We use the Bernoulli part extraction displayed at Lemma \ref{lemd},  \eqref{dec0}, \eqref{dec} as well as the notation introduced. Let 
 \begin{eqnarray}\label{dep0}A_n=\big\{B_n\le  (1-\e)\nu_n
\big\}     .
\end{eqnarray}
 We deduce from Lemma \ref{di.1} that  $\P\{A_n \} \, \le  e^{- \frac{\epsilon^2\nu_n}{2}}$ for all positive $n$. We write
\begin{equation}\label{dep..} \P \{d|S_n  \} -{1\over d}     \,=\, \E_{(V,\e)}  
  \, \big( \chi(A_n)+\chi(A_n^c)\big) 
 \,  \,  \Big(\P_{\!L}
 \big\{d|\big(  \sum_{j= 1}^n \e_jL_j+W_n  \big)
\big\}-{1\over d}\Big) 
 . \end{equation}

  On the one hand, 
\begin{eqnarray}\label{proof.th.saud1}& & \E_{(V,\e)} 
   \chi(A_n) 
 \, \Big|    \P_{\!L}
\big\{d|\big(  \sum_{j= 1}^n \e_jL_j+W_n  \big)
\big\}
    -{1\over d}
\Big|\ \le \,2 \P\{A_n \} \, \le  2 e^{- \frac{\epsilon^2 }{2}\nu_n}.
\end{eqnarray}
 So that
  \begin{equation}\label{dep1} \big|\P \{d|S_n  \} -{1\over d}   \big|  \,\le \,2 e^{- \frac{\epsilon^2 }{2}\nu_n}+ \E_{(V,\e)}  
  \, 
  \chi(A_n^c) \, \cdot \,  \Big|\P_{\!L}
 \big\{d|\big(  \sum_{j= 1}^n \e_jL_j+W_n  \big)
\big\}-{1\over d}\Big| 
 . \end{equation}
 
 Now on $A_n^c$, $B_n\ge (1-\epsilon)\nu_n $, and  since  $ \sqrt{   x /  \log x}$ is increasing on $[e,\infty)$, we have 
\begin{equation}\label{phintaun1}
 \sqrt{{  (1-\epsilon)\nu_n \over 2\a\log (1-\epsilon)\nu_n}}\le \sqrt{{   B_n \over 2\a\log B_n}}.
 \end{equation}
Also 
\begin{equation}\label{phintaun2}  \p_n=\sqrt{\frac {2\a\log B_n}{   B_n}}\le \sqrt{  \frac{ 2\a\log (1-\epsilon)\nu_n}{ (1-\epsilon)\nu_n }} \quad \hbox{\rm and \ thus} \quad {\sin\p_n/2\over
\p_n /2}\ge (\a^\prime/\a)^{1/2},
\end{equation}
 by the assumption made.
 \vskip 2 pt  By applying    Proposition \ref{special.cases}, we have $\P_{(V,\e)}$ almost surely,
\begin{equation*}
 \sup_{u\ge 0}\,\sup_{d<  \pi   \sqrt{   B_n \over 2\a\log B_n}}\Big|\P_{\!L}
\Big\{d\,\big|\Big(  \sum_{j=1}^{B_n } L_j+W_n +u \Big)
\Big\}-{1\over d} 
\Big|\,\le\, B_n^{-\a'}.\end{equation*}
Whence on $A_n^c$,
   \begin{eqnarray}\label{proof.th.saud2}& & \sup_{u\ge 0}\,\sup_{d<  \pi  \sqrt{   (1-\epsilon)\nu_n \over 2\a\log (1-\epsilon)\nu_n}}\Big|\P_{\!L}
\Big\{d\,\big|\Big(  \sum_{j=1}^{B_n } L_j+W_n +u \Big)
\Big\}-{1\over d} 
\Big|
\cr &\le&    \sup_{u\ge 0}\,\sup_{d<  \pi   \sqrt{   B_n \over 2\a\log B_n}}\ \Big|\P_{\!L}
\Big\{d\,\big|\Big(  \sum_{j=1}^{B_n } L_j+W_n +u \Big)
\Big\}-{1\over d} 
\Big|\cr &\le& B_n^{-\a'}\, \le \,\big( (1-\epsilon)\nu_n\big)^{-\a'} .
\end{eqnarray}

In view of \eqref{dep1} and \eqref{proof.th.saud2}, we get  for all $u\ge 0$ and  $d<  \pi  \sqrt{{   (1-\epsilon)\nu_n \over 2\a\log (1-\epsilon)\nu_n}}$, 
  \begin{eqnarray} 
  \big| \P \{d|S_n+u  \}   -  {1\over d} \big|
   &\le  & 2 e^{- \frac{\epsilon^2 }{2}\nu_n}+  \,\big( (1-\epsilon)\nu_n\big)^{-\a'} \E_{(V,\e)} 
  \,    \chi(A_n^c) 
 \cr &\le &2 e^{- \frac{\epsilon^2 }{2}\nu_n}+
 \,\big( (1-\epsilon)\nu_n\big)^{-\a'}  .
\end{eqnarray} 
\end{proof}

\vskip 3 pt 
\vskip 3 pt 
The next result shows a considerable variation of the  speed of convergence when $d$ is less close to $\sqrt{\nu_n}$.

 \begin{theorem}\label{saud2} Let $0<\rho<1 $ and   $0<\e<1$.Then for   each $n$ such that 
$$|x|\le\frac12 \,\sqrt{  \frac{ 2 }{ ((1-\epsilon)\nu_n)^{1-\rho} }}\qq \Rightarrow \qq{\sin x\over
x}\ge \sqrt{1-\e}$$
we have\begin{eqnarray*} 
\sup_{u\ge 0}\,\sup_{d<  (\pi/\sqrt 2) ((1-\e)\nu_n)^{(1-\rho)/2} }\ \big| \P \{d|S_n+u  \}   -  {1\over d} \big|  
   &\le  &  2 e^{- \frac{\epsilon^2 }{2}\nu_n}+e^{- ( (1-\epsilon)\nu_n)^\rho}
  .
\end{eqnarray*}
\end{theorem}

\begin{proof} The proof is similar. We operate with the same set $A_n$ as in \eqref{dep0}, and  use the decomposition   \eqref{dep}.
Let $0<\rho<1 $ and $0<\e<1$.
  
  By applying    Proposition \ref{special.cases} with $\eta=\e$, we have $\P_{(V,\e)}$ almost surely, for 
    $n$ such  that $\widetilde\tau_n\ge \sqrt{1-\e}$, 
    where here 
$$  \widetilde\tau_n= {\sin\psi_n/2\over
\psi_n /2}\qq {\rm with}\qq \psi_n= \big({2B_n^\rho \over B_n}\big)^{1/2},$$ 
 \begin{equation*}  \sup_{u\ge 0}\,\sup_{d<  (\pi/\sqrt 2) B_n^{(1-\rho)/2} }\Big|\P_{\!L}
\Big\{d\,\big|\Big(  \sum_{j=1}^{B_n } L_j+W_n +u \Big)
\Big\}-{1\over d} 
\Big|\,\le\, e^{-(1-\e) B_n^\rho}.
\end{equation*}
 
By using corresponding estimates to \eqref{phintaun1}, \eqref{phintaun2}, namely that  on $A_n^c$,
$$\psi_n=\Big(\frac{2}{B_n^{1-\rho}}\Big)^{1/2}\le \Big(\frac{2}{((1-\e)\nu_n)^{1-\rho}}\Big)^{1/2}, $$
so that $\widetilde\tau_n\ge \sqrt{1-\e}$, we   deduce that on 
 $A_n^c$,$$  \sup_{u\ge 0}\,\sup_{d<  (\pi/\sqrt 2) ((1-\e)\nu_n)^{(1-\rho)/2} }\  \Big|\P_{\!L}
\Big\{d\,\big|\Big(  \sum_{j=1}^{B_n } L_j+W_n +u \Big)
\Big\}-{1\over d} 
\Big|
$$ $$\,\le\,    \sup_{u\ge 0}\,\,\sup_{d<  (\pi/\sqrt 2) B_n^{(1-\rho)/2} }\ \Big|\P_{\!L}
\Big\{d\,\big|\Big(  \sum_{j=1}^{B_n } L_j+W_n +u \Big)
\Big\}-{1\over d} 
\Big| \,\le\, e^{-(1-\e) B_n^\rho}.$$

  Therefore
 \begin{eqnarray} && \sup_{u\ge 0}\,\sup_{d<  (\pi/\sqrt 2) ((1-\e)\nu_n)^{(1-\rho)/2} }\ \big| \P \{d|S_n+u  \}   -  {1\over d} \big| \cr  &\le  & 2 e^{- \frac{\epsilon^2 }{2}\nu_n}+   \E_{(V,\e)} 
  \,    \chi(A_n^c) 
 \,e^{-(1-\e) B_n^\rho}
\,\le \, 2 e^{- \frac{\epsilon^2 }{2}\nu_n}+e^{-(1-\e)^{1+\rho} \nu_n ^\rho}
  .
\end{eqnarray}

\end{proof}

\begin{remark} So far we only have considered necessary conditions  for the validity of the local limit theorem, which are  formulated in terms of a.u.d. property, as well as strenghtenings of this property yielding effective speed of convergence bounds. It is important to mention in that context, that in 1984, Mukhin found  a remarkable necessary and sufficient condition for the validity of the local limit theorem.  
Let $\{S_n,n\ge 1\}$ be a sequence of  $\Z$--valued random variables  such  that an integral limit theorem  holds:  there exist $a_n\in \R$ and real $b_n\to \infty$ such that the sequence of distributions of $(S_n-a_n)/b_n$ converges weakly to an absolutely continuous distribution $G$ with density $g(x)$, which is uniformly continuous in $\R$.
    The  local limit theorem is   valid if
\beq\label{ilt.llt}
\P\{S_n=m\}=B_n^{-1} g\Big(\frac{m-A_n}{B_n}\Big) + o(B_n^{-1}),
\eeq
uniformly in $m\in \Z$. Muhkin showed that the validity of the local limit theorem is equivalent to the existence
 of a sequence of integers $v_n=o(b_n)$ such that
\beq\label{ilt.llt.diff} \sup_{m}\Big|\P\{S_n=m+v_n\big\}-\P\{S_n=m \big\}\Big|\,=\,\,o\Big(\frac{1}{b_n}\Big).
\eeq
 
Revisiting the succint proof given in   \cite{Mu2}, we however could only prove rigorously a weaker necessary and sufficient condition,   with a significantly different formulation, namely that a necessary and sufficient condition for the local limit theorem in the usual form to hold is 
\beq \sup_{m, k\in\Z\atop |m-k|\le \max\{1, [\sqrt \e_n b_n]\}}\Big|\P\{S_n=m\big\}-\P\{S_n=k \big\}\Big|\,=\,\,o\Big(\frac{1}{b_n}\Big),
\eeq
where \beq \label{en}\e_n:=\sup_{x\in \R}\Big|\P\Big\{\frac{S_n-a_n}{b_n}<x\Big\}-G(x)\Big|\ \to \ 0,
\eeq
by the integral limit theorem. 
This is  the object of  the    Note \cite{W5},  with  remarks and   references on general relations of type \eqref{ilt.llt.diff} therein.  Mukhin wrote at this regard in \cite{Mu2}: \lq\lq ... getting from here more general sufficient conditions turns out to be difficult in view of the lack of good criteria.  Working with asymptotic equidistribution properties are more convenient in this respect\,\rq\rq.  

\end{remark}



\appendix



\vskip 6pt 
\section{LLT's with speed of convergence.} Let $S_n=X_1+\ldots +X_n$, $n\ge 1$, where
$X_j$ are independent random variables such that
 $\P\{X_j
\in\mathcal L(v_{ 0},D )\}=1$. 
\vskip 3 pt
Assume first that the random variables $X_j$ are identically distributed. Then we have the following characterization result.
\begin{theorem} \label{r}    Let
$F$ denote the distribution function of
$X_1$.

{\rm  (i) (\cite{IBLIN}, Theorem 4.5.3)}      In order that  the property
\begin{equation} \label{alfa}
 \sup_{N=an+Dk}\Big|
 { \frac{\s \sqrt n}{D} }{\mathbb P}\{S_n=N\}-{1 \over  \sqrt{ 2\pi}\s}e^{-
{(N-n\m )^2\over  2 n \s^2} }\Big| ={\mathcal O}\big(n^{-\alpha{/2}} \big) ,
  \end{equation}
 { where $0<\a<1$},
 it is necessary and sufficient that the following conditions be satisfied:
 \begin{eqnarray*} (1) \   D \ \hbox{is maximal}, \ \qq\qq
(2)  \  \  \int_{|x|\ge u} x^2 F(dx) = \mathcal O(u^{-\a})\quad \hbox{as $u\to \infty$.}
\end{eqnarray*}

   {\rm (ii) (\cite{P} Theorem 6 p.\,197)} If ${\mathbb E\,} |X_1|^3<\infty$, then \eqref{alfa} holds with $\a =1/2$.

\end{theorem}
 \vskip 6 pt
Now consider the non-identically distributed case. Assume  that (see \eqref{vartheta})
\begin{equation}\label{basber.pos}  \t_{X_j}>0, \qq \quad  j=1,\ldots, n.
\end{equation}
Let $\nu_n =\sum_{j=1}^n \t_j$.   Let $\psi:\R\to \R^+$ be even, convex and such that   $\frac
{\psi(x)}{x^2}$  and $\frac{x^3}{\psi(x)}$  are non-decreasing on $\R^+$. We further assume that
  \begin{equation}\label{did}  \E \psi( X_j )<\infty    .
 \end{equation} Put $$L_n=\frac{  \sum_{j=1}^n\E \psi (X_j)  }
{   \psi (\sqrt
{ {\rm Var}(S_n )})}  .$$
 The following result is  Corollary 1.7 in Giuliano-Weber  in \cite{GW3}.
   \begin{theorem}\label{ger3} Assume that $\frac{  \log \nu_n }{\nu_n}\le  {1}/{14} $. Then, for all $\k\in \mathcal L( v_{
0}n,D )$ such that
$$\frac{(\k- \E
S_n)^2}{    {\rm Var}(S_n)  } \le \sqrt{\frac{7 \log \nu_n} {2\nu_n}},$$  we have
\begin{eqnarray*}   \Big| \P \{S_n =\kappa \} -{ D e^{- \frac{(\k- \E
S_n)^2}{    2 {\rm Var}(S_n)    } } \over \sqrt{2\pi {\rm Var}(S_n)     }}     \Big|         & \le &    C_3\Big\{
D\big({    {   \log \nu_n }     \over
 {    {\rm Var}(S_n)   \nu_n} } \big)^{1/2}  +    {    L_n
 +  \nu_n^{-1}
\over \sqrt{   \nu_n} } \Big\} .
   \end{eqnarray*}
And $C_3=\max (C_2, 2^{  3/2}C_{{\rm E}}) $,   $C_{{\rm E}}$ being an absolute constant arising from  Berry-Esseen's inequality.
\end{theorem}

\vskip 8 pt We pass to another speed of convergence result due to Mukhin.   Consider the structural   characteristic of a random variable $X$, introduced and studied by Mukhin in \cite{Mu1} and \cite{Mu} for instance,
$$ H(X ,d) = \E \langle X^*d\rangle^2,$$  where  $\langle \a \rangle$ denotes the distance from $\a$ to the nearest integer, and $X^*$
is   a symmetrization of
$X$. Let $\p_X$ be the characteristic function $X$. 
   The two-sided inequality
 \begin{eqnarray}\label{fih} 1-2\pi^2 H(X ,\frac{t }{2\pi})  \le |\p_X(t)|\le 1-4  H(X ,\frac{t }{2\pi})   ,
\end{eqnarray}
is established in the above references. See also    Szewczak   and  Weber  \cite{SW}  for more.

The following   is  the one-dimensional version of Theorem 5 in \cite{Mu}, see also \cite{SW}  and is stated without proof, however.

\begin{theorem}[Mukhin]\label{Mukhin.th.Hn} Let $X_1,\ldots, X_n$ have zero mean and finite third moments. Let
$$ B_n^2= \sum_{j=1}^n{\mathbb E\,} |X_j|^2 ,\qq H_n= \inf_{1/4\le d\le 1/2}\sum_{j=1}^n H(X_j
,d), \qq L_n= \frac{\sum_{j=1}^n{\mathbb E\,} |X_j|^3}{(B_n)^{3/2}} .$$ Then
   \begin{equation}\label{llt}    \sup_{N=v_0n+Dk }\Big|B_n {\mathbb P}\{S_n=N\}-{D\over  \sqrt{ 2\pi } }e^{-
{(N-M_n)^2\over  2 B_n^2} }\Big|\,\le CL_n\, \big( {B_n }/{ H_n}\big) .
\end{equation} \end{theorem}

\end{document}